\documentclass[11pt]{amsart}
\usepackage{url}
\usepackage[all]{xypic}
\usepackage{amsmath,amssymb,amsfonts,amsthm}
\newtheorem{thm}{Theorem}[section]
\newtheorem{theorem}[thm]{Theorem}
\newtheorem {definition}[thm]{Definition}

\newtheorem{corollary}[thm]{Corollary}
\newtheorem{lemma}[thm]{Lemma}

\newtheorem{proposition}[thm]{Proposition}

\newtheorem{remark}[thm]{Remark}

\numberwithin{equation}{section}
\begin{document}
%\baselineskip=16pt
%\noindent Research project
%\newline
%\title{Study invariants of singularities in positive characteristic}   

\title[Invariants of plane curve singularities and Pl\"ucker formulas]{Invariants of plane curve singularities and Pl\"ucker formulas in positive characteristic}      
\author{Nguyen Hong Duc}
\address{Institute of Mathematics, Vietnam Academy of Science and Technology\newline \indent18 Hoang Quoc Viet Road, Cau Giay District, \newline \indent  10307, Hanoi.} 
\email{nhduc@math.ac.vn}
\date{\today}                  

%\tableofcontents
\begin{abstract}
We study classical and new invariants of plane curve singularities $f\in K[[x,y]]$, $K$ an algebraically closed field of characteristic $p\geq 0$. It is known, in characteristic zero, that $\kappa(f)=2\delta(f)-r(f)+\mathrm{mt}(f)$, where $\kappa(f), \delta(f),r(f)$ and $\mathrm{mt}(f)$ denotes kappa invariant, delta invariant, number of branches and multiplicity of $f$ respctively. For arbitrary characteristic, by introducing new invariant $\gamma$, we prove in this note that $\kappa(f)\geq \gamma(f)+\mathrm{mt}(f)-1\geq 2\delta(f)-r(f)+\mathrm{mt}(f)$ with equalities if and only if the characteristic $p$ does not divide the multiplicity of any branch of $f$. As applications we obtain some Pl\"ucker formulas for projective plane curves in positive characteristic. Moreover we show that if $p$ is ``big'' for $f$, resp. for irreducble curve $C\subset \mathbb{P}^2$ (in fact, if $p > \kappa(f)$, resp. $p > \deg C(\deg C-1)$), then $f$, resp. $C$ has no wild vanishing cycle.  
\end{abstract}
\maketitle
\section{\bf Introduction}
Let $K$ be an algebraically closed field of characteristic $p\geq 0$ and $K[[x,y]]$ the ring of formal power series. We study classical and new invariants: Milnor number, kappa invariant, delta invariant, Swan conductor, multiplicity and their relations  (see Section \ref{sec2} for definitions and facts). 

Let us recall some notions and facts on plan curve singularities (see \cite{GLS06} and \cite{Ng13} for proofs). Let $f,g\in K[[x,y]]$. We denote by $i(f,g):=\dim K[[x,y]]/\langle f,g\rangle$ the {\em intersection multiplicity} of $f$ and $g$. The {\em Milnor number} (resp. {\em Kappa invariant}) of $f$ is the intersection multiplicity $i(f_x,f_y)$ (resp. $i(f, \alpha f_x+\beta f_y)$), where $f_x,f_y$ are the partials of $f$ and $(\alpha:\beta)\in \mathbb{P}^1$ is generic. 

If $f$ is reduced and $\bar{R}$ is a normalization of $R:=K[[x,y]]/\langle f\rangle$, then the dimension $\dim \bar{R}/R$ is called the  {\em delta invariant} of $f$, denoted by $\delta(f)$. It relates to the Milnor number by the Milnor formula (see \cite{Mil68}) stating that, if $\mathrm{char}(K)=0$ then
$$\mu(f)=2\delta(f)-r(f)+1,$$
where $r(f)$ is the number of branches of $f$. This does not hold in positive characteristic because of the existence of wild vanishing cycles. More precisely, using \'etale cohomology, Deligne showed that the Milnor number $\mu$ (resp. $2\delta(f)-r(f)+1$, resp. $\mathrm{Sw}(f)$) is equal to the number of total (resp. ordinary, resp. wild) vanishing cycles (of the Milnor fiber) of $f$ (cf. \cite{Del73}, \cite{M-HW01}), where
$\mathrm{Sw}(f)$ denotes the Swan conductor of $f$ (see \cite{Del73}, 1.7, 1.8 for the definition). 
%i.e. the alternative sum of the Swan conductor of vector spaces ${\Phi_{\bar{\eta}}^i}(\mathbb{Z}/l\mathbb{Z})_0$, where $\Phi_{\bar{\eta}}^i$. 
This implies that
$$\mu(f)= 2\delta(f)-r(f)+1+\mathrm{Sw}(f)$$
and therefore
$$\mu(f)\geq 2\delta(f)-r(f)+1.$$
However it is unknown how a plane curve singularity without wild vanishing cycle can be reasonably characterized. In this paper we will give a partial answer for this problem saying that if the characteristic $p$ is ``big'' for $f$ (resp. for irreducible projective curve $C$) (e.g. $p > \kappa(f)$, resp. $p > \deg C(\deg C-1)$), then $f$, resp. $C$ has no wild vanishing cycle.  (Corollary \ref{coro33}). %However , we do not have any geometric explanation for this.

Delta invariant is also related to kappa invariant by the following formula in characteristic zero (cf. \cite{GLS06}):
$$\kappa(f)=2\delta(f)-r(f)+\mathrm{mt}(f),$$
where $\mathrm{mt}(f)$ is the {\em multiplicity} of $f$, defined to be the maximal of $k$ for which $\langle f\rangle\subset \langle x,y\rangle^k$. It is easy to see that the equality does not hold in positive characteristic. %Kappa invariants occur in the first Pl\"ucker formula which concerns the dual curve and its degree of an algebraic projective plane curve (see ). 
The aim of this paper is to see how is this relation in positive characteristic. We introduce and study new invariants, $\gamma, \tilde{\gamma}$ and prove in Theorem \ref{thm21}, that
$$\gamma(f)\geq 2\delta(f)-r(f)+1$$
with equality if and only if $p$ is {\em right intersection multiplicity good} for $f$ (i.e. there exist a coordinate $X,Y$ such that for any branch $f_i$ of $f$, $p$ does not divide at least one of $i(f_i,X)$ and $i(f_i,Y)$, see Definition \ref{def22}). We then obtain the main result of the present paper (Theorem \ref{thm31}) stating that there is always the inequality
$$\kappa(f)\geq 2\delta(f)-r(f)+\mathrm{mt}(f)$$
with equality if and only if $p$ is {\em multiplicity good} for $f$ (i.e. it does not divide the multiplicity of any branch of $f$, see Definition \ref{def22}).

We now apply our main result to the Pl\"ucker formulas. Recall that the first Pl\"ucker formula gives a relation between the degree $d$ of an irreducible curve $C\subset \mathbb P^2$ and the degree $\check{d}$ of its dual curve  (cf. \cite{Wal62}, \cite{Pie78}). Precisely, one has (see \ref{eq31})
$$d(d-1)=\deg\rho\cdot\check{d}+\sum_{P\in\mathrm{Sing}(C)}\kappa(f_P),$$
where $\deg\rho$ is the degree of the dual map $\rho$, $\mathrm{Sing}(C)$ is the singular locus of $C$ and $f_P=0$ is a local equation of $C$ at $P$.
%Recall that two power series $f, g \in K[[x,y]]$ are {\em right equivalent}, denoted this by $f \sim_r g$, if there is an automorphism (coordinate change) $\Phi\in Aut_K (K[[x,y]])$ such that $f = \Phi(g)$. We call $f, g \in K[[x,y]]$ {\em contact equivalent} if there is an automorphism $\Phi\in Aut_K (K[[x,y]])$ and a unit $u\in K[[x,y]]$ such that $f = u\cdot \Phi(g)$, and we denote this by $f \sim_c g$.
%An important invariant of plane curve singularities which appears in  is the kappa invariant defined by the following intersection multiplicity (see Section \ref{sec2})$$\kappa(f):=i(f,\alpha f_x+\beta f_y),$$where $\alpha f_x+\beta f_y$ is a generic polar of $f$. . 
%Moreover kappa invariant is related to delta invariant . 
Applying Theorem \ref{thm31} we obtain a kind of Pl\"ucker formula for an irreducible plane curve $C$ in positive characteristic stating that one has
\begin{eqnarray*}
\deg(\rho)\cdot\check{d}&\leq& d(d-1)-2\delta(C)+r(C)-\mathrm{mt}(C)
\end{eqnarray*}
with equality if and only if $p$ is {\em multiplicity good} for $C$ (i.e. it is multiplicity good for all local functions $f_P$ of $C$ at singular points $P$), where $\delta(C)$ (resp. $r(C)$, resp. $\mathrm{mt}(C)$) is the sum of the delta invariants (resp. number of branches, resp. multiplicities) of $f_P$ (see Corollary \ref{coro34}). We show further that if $p$ is ``big'' for $C\subset \mathbb{P}^2$ (e.g. $p > d(d-1)$), then one has
\begin{eqnarray*}
\deg(\rho)\cdot\check{d}&=& d(d-1)-2\delta(C)+r(C)-\mathrm{mt}(C)\\
&=& d(d-1)-\mu(C)-\mathrm{mt}(C)+s(C),
\end{eqnarray*}
where $\mu(C)$ denotes the sum of the local Milnor numbers $\mu(f_P)$ and $s(C)$ the number of singular points of $C$.

The paper is organized as follows. In Section \ref{sec2} we introduce and study a new invariant (gamma invariant) and its relation to classical invariants of plane curve singularities which play an important role in the proof of the main result. We present and prove the main result and its applications to Pl\"ucker formula in Section \ref{sec3}. Our method is based on resolution and parametrization of plane curve singularities. 
\subsection*{Acknowledgement} 
%We would like to thank the referees for their careful reading of the manuscript and helpful comments which improved the presentation of this paper. 
A part of this article was done in my thesis \cite{Ng13} under the supervision of Professor Gert-Martin Greuel at the Technische Universit\"at Kaiserslautern. I am grateful to him for many valuable suggestions. The author's research is funded by Vietnam National Foundation for Science and Technology Development (NAFOSTED) under Grant Number 101.04-2014.23.

\section{\bf Gamma invariants}\label{sec2}
We introduce and study new (gamma) invariants $\gamma,\tilde{\gamma}$ of plane curve singularities which have not been considered before. In characteristic zero, these invariants coincide and are equal to the Milnor number (see Remark \ref{rm21}). So they may be considered as generalizations of the Milnor number in positive characteristic and are believed to be useful in studying classical invariants. In this section we use them to connect the delta and kappa invariant. We will show, in Proposition \ref{pro21}, that
$$\kappa(f)\geq \gamma(f)+\mathrm{mt}(f)-1$$
and in Theorem \ref{thm22}, that
$$\gamma(f)\geq 2\delta(f)-r(f)+1$$
and obtain the inequality in the main result of the paper (Theorem \ref{thm31}). 

\begin{definition}\label{def21}{\rm
Let $f\in K[[x,y]]$ be reduced. The number $\tilde{\gamma}_{x,y}(f)$ (or $\tilde{\gamma}(f)$, if the coordinate $\{x,y\}$ is fixed) of $f$, is defined as follows: 
\begin{itemize}
\item[(1)] $\tilde{\gamma}(x):=0$, $\tilde{\gamma}(y):=0$.
\item[(2)] If $f$ is irreducible and {\em convenient} (i.e. $i(f,x),i(f,y)<\infty$), then
$$\tilde{\gamma}(f):=\min\{i(f,f_x)-i(f,y)+1, i(f,f_y)-i(f,x)+1\}.$$
\item[(3)] If $f=f_1\cdot\ldots\cdot f_r$, then
$$\tilde{\gamma}(f):=\sum_{i=1}^r\Big(\tilde{\gamma}(f_i)+\sum_{j\neq i} i(f_i,f_j)\Big)-r+1.$$
\end{itemize}
}\end{definition}
\begin{definition}\label{def211}{\rm
The {\em gamma invariant} of a reduced plane curve singularity $f$, denoted by $\gamma(f)$, is the minimum of $\tilde{\gamma}_{X,Y} (f)$ for all coordinates $X,Y$.
}\end{definition}
\begin{remark}\label{rm21}{\rm
(a) In characteristic zero, $\gamma(f)=\tilde{\gamma}(f)=\mu(f)$ due to Theorems \ref{thm21}, \ref{thm22} and the Milnor formula. 

(b) In general we have, by definition, that $\gamma(f)\leq \tilde{\gamma}(f)$ (with equality if $p$ is im-good for $f$, see Corollary \ref{coro21}) and that $\gamma(f)= \tilde{\gamma}(g)$ for some $g$ right equivalent to $f$ ($f$ is called {\em right equivalent} to $g$, denoted by $f \sim_r g$, if there is an automorphism $\Phi\in Aut_K (K[[x,y]])$ such that $f = \Phi(g)$).

(c) The number $\tilde{\gamma}$ depends on the choice of coordinates, i.e. it is not invariant under right equivalence . E.g. $f=x^3+x^4+y^5$ and $g=(x+y)^3+(x+y)^4+y^5$ in $K[[x,y]]$ with $\mathrm{char} (K)=3$ and then $f\sim_r g$, but $\tilde{\gamma}(f)=8$, $\tilde{\gamma}(g)=10$. However, as we will see in Proposition \ref{pro22}, if the characteristic $p$ is multiplicity good for $f$ then $\tilde{\gamma}(f)=\tilde{\gamma}(g)$ for all $g$ contact equivalent to $f$. Recall that $f, g$ are {\em contact equivalent} if there is an automorphism $\Phi\in Aut_K (K[[x,y]])$ and a unit $u\in K[[x,y]]$ such that $f = u\cdot \Phi(g)$, and we denote this by $f \sim_c g$.

(d) It follows from the definition that $\tilde{\gamma}(u)=1$ and $\tilde{\gamma}(u\cdot f)=\tilde{\gamma}(f)$ for every unit $u$ and therefore $\gamma$ is invariant under contact equivalence.

(e) Milnor number $\mu$ is invariant under right equivalence. The numbers $\delta, \kappa, \mathrm{mt}, r,i$ are invariant under contact equivalence (see, for instance \cite{Ng13}, Prop. 1.2.19 for the invariance of $\delta$). This means that, if $f\sim_c g$ then
$$\delta(f)=\delta(g),\ \kappa(f)=\kappa(g),\ \mathrm{mt}(f)=\mathrm{mt}(g)\text{ and } r(f)=r(g).$$
Moreover, for any $\Phi\in Aut_K\left(K[[x,y]]\right)$ and units $u,v$, one has
$$i(f,h)=i\left(u\cdot\Phi(f),v\cdot\Phi(h)\right).$$
}\end{remark}
Before studying in detail gamma invariants, we collect several facts on invariants of plane curve singularities which we use later. For proofs, we refer to \cite{GLS06} and \cite{Ng13}.
\begin{remark}\label{rm22}
{\rm 
(a) If $f\in K[[x,y]]$ is irreducible then there exists a couple $(x(t),y(t))\in K[[t]]^2$ such that $f(x(t),y(t))=0$ and it satisfies the following universal property: for each $(u(t),v(t))\in K[[t]]^2$ with $f(u(t),v(t))=0$, there exists a unique series $h(t)\in K[[t]]$ such that $u(t)=x(h(t)) \text{ and } v(t)=y(h(t))$. We call such a couple a {\em parametrization} of $f$. The order of its parametrization is related to the multiplicity of a singularity as follows $\mathrm{mt}(f)=\min\{\mathrm{ord}\ x(t), \mathrm{ord}\ y(t)\}$. Here for each univariate formal power series $\varphi(t)$, $\mathrm{ord}\ \varphi(t)$ denotes the order (or multiplicity, or valuation) of $\varphi(t)$. For each formal series $g$, one has $i(f,g)=\mathrm{ord}\ g(x(t),y(t))$ (cf. \cite{GLS06}, Chap. 1, Prop. 3.12).

(b) If $f$ is irreducible, then
$$\kappa(f)=\min\{i(f,f_x),i(f,f_y)\}.$$
Indeed, taking a parametrization $(x(t),y(t))$ of $f$ we obtain that
\begin{eqnarray*}
\kappa(f)=\mathrm{ord}\ \left(\alpha f_x(x(t),y(t))+\beta f_y(x(t),y(t))\right),
\end{eqnarray*}
which equals to the minimum of $i(f,f_x)$ and $i(f,f_y)$ since $(\alpha:\beta)$ is generic.

(c) If $f$ is convenient, then
\begin{eqnarray*}
\tilde{\gamma}(f)&=&i(f,\alpha xf_x+\beta yf_y)-i(f,x)-i(f,y)+1,
\end{eqnarray*}
where $(\alpha:\beta) \in \mathbb P^1$ is generic. In fact, assume first that $f$ is irreducible and take a parametrization $(x(t),y(t))$ of $f$. Then one has
\begin{eqnarray*}
\tilde{\gamma}(f)&=&\min\{i(f,f_x)-i(f,y)+1, i(f,f_y)-i(f,x)+1\}\\
&=&\min\{i(f,f_x)+i(f,x), i(f,f_y)+i(f,y)\}-i(f,f_y)-i(f,x)+1\\
&=&\min\{\mathrm{ord}\ x(t)f_x(x(t),y(t)),\mathrm{ord}\ y(t)f_y(x(t),y(t))\}-i(f,f_y)-i(f,x)+1\\
&=&i(f,\alpha xf_x+\beta yf_y)-i(f,x)-i(f,y)+1
\end{eqnarray*}
with $(\alpha:\beta) \in \mathbb P^1$ generic. The reducible case is thus followed by using Definition \ref{def21} and simple calculations (for more detail, see \cite{Ng13}, Lemma 2.7.2(ii)). 

(d) If $f=f_1\cdot\ldots\cdot f_r$, then
\begin{eqnarray*}
i(f,g)&=& \sum_{i=1}^r i(f_i,g),\ \forall g\in K[[x,y]];\\
2\delta(f)&=&\sum_{i=1}^r\Big(2\delta(f_i)+\sum_{j\neq i} i(f_i,f_j)\Big);\\
\kappa(f)&=&\sum_{i=1}^r\Big(\kappa(f_i)+\sum_{j\neq i} i(f_i,f_j)\Big).
\end{eqnarray*}
Here the two first equalities follow from \cite{GLS06}, Chap. 1, Prop. 3.12 and Lemma 3.32 respectively and the last one is done by simple caculations.
}\end{remark}
\begin{definition}\label{def22}{\rm
Let $\mathrm{char}(K)=p\geq 0$ and let $f=f_1\cdot\ldots\cdot f_r\in K[[x,y]]$ be reduced with $f_i$ irreducible. The characteristic $p$ is said to be
\begin{itemize}
\item[1.] {\em multiplicity good} ({\em m-good}) for $f$ if the multiplicities $\mathrm{mt}(f_i)\neq 0$ (mod $p$) for all $i=1,\ldots,r$; 
\item[2.] {\em intersection multiplicity good} ({\em im-good}) for $f$ if for all $i=1,\ldots,r$, either\\
\centerline{$ i(f_i,x)\neq 0$ (mod $p$) or $i(f_i,y)\neq 0$  (mod $p$);}
\item[3.] {\em right intersection multiplicity good} ({\em right im-good}) for $f$ if it is im-good for $f$ after some change of coordinate. That is, it is im-good for some $g$ right equivalent to $f$. 
\end{itemize}
}\end{definition}
Note that these notions are trivial in characteristic zero, i.e. if $p=0$ then it is always m-good, im-good and right im-good for $f$. In general we have
$$\text{``m-good"} \Longrightarrow  \text{``im-good"} \Longrightarrow  \text{``right im-good"}.$$
%\begin{lemma}\label{lm22}
%Let $f\in K[[x,y]]$ be such that $i(f,x)=\mathrm{mt}(f)$. Then $$\kappa(f)\leq i(f,f_y)$$ with equality if $p$ is m-good for $f$.
%\end{lemma}
The following proposition gives us the first relations between the gamma invariants and classical invariants.
\begin{proposition}\label{pro21}
Let $f\in K[[x,y]]$ be reduced. Then
$$\gamma(f)\leq\tilde{\gamma}(f)\leq \kappa(f)-\mathrm{mt}(f)+1$$ with equality if $p$ is m-good for $f$.
\end{proposition}
\begin{proof}
The first inequality is trivial. The proof of the second inequality (and equality) will be divided into two steps:\\
{\bf Step 1}: Suppose that $f$ is irreducible. 

One has
\begin{eqnarray*}
\tilde{\gamma}(f)&=&\min\{i(f,f_x)-i(f,y)+1, i(f,f_y)-i(f,x)+1\}\\
&\leq & \min\{i(f,f_x)-\mathrm{mt}(f)+1, i(f,f_y)-\mathrm{mt}(f)+1\}\\
&=& \min\{i(f,f_x), i(f,f_y)\}-\mathrm{mt}(f)+1\\
&=& \kappa(f)-\mathrm{mt}(f)+1.
\end{eqnarray*}

Assume that $p$ is m-good for $f$. Let $(x(t),y(t))$ be a parametrization of $f$. We may assume that $\mathrm{ord\ }x(t)\leq \mathrm{ord\ }y(t)$, so that $m=\mathrm{ord\ }x(t)\leq \mathrm{ord\ }y(t)$. Since $f(x(t),y(t))=0$,
$$f_x(x(t),y(t))\cdot x'(t)+f_y(x(t),y(t))\cdot y'(t)=0.$$
Therefore,
$$i(f,f_x)+\mathrm{ord} x'(t)=i(f,f_y)+\mathrm{ord} y'(t).$$
It follows that
$$i(f,f_x)-\mathrm{ord} y(t)\geq i(f,f_y)-\mathrm{ord} x(t)=i(f,f_y)-m,$$
since $\mathrm{ord} x'(t)=\mathrm{ord} x(t)-1=m-1$ and $\mathrm{ord} y'(t)\geq \mathrm{ord} y(t)-1$. This implies that
$$\tilde{\gamma}(f)=i(f,f_y)-m+1\text{ and } i(f,f_y)\leq i(f,f_x).$$
Hence
\begin{eqnarray*}
\tilde{\gamma}(f)&=&\min\{i(f,f_x)-i(f,y)+1, i(f,f_y)-i(f,x)+1\}\\
&= &i(f,f_y)-m+1\\
&=& \min\{i(f,f_x), i(f,f_y)\} -m+1\\
&=& \kappa(f)-\mathrm{mt}(f)+1.
\end{eqnarray*}
{\bf Step 2}: The general case as $f=f_1\cdot\ldots\cdot f_r$ with $f_i$ irreducible, follows from the first step and Remark \ref{rm22}.

It remains to prove that if $p$ is m-good for $f$ then $\gamma(f)=\tilde{\gamma}(f)$. Indeed, take $g$ right equivalent to $f$ such that $\gamma(f)=\tilde{\gamma}(g)$. Since $p$ is also m-good for $g$, it follows from the second equality that
$$\gamma(f)=\tilde{\gamma}(g)=\kappa(g)-\mathrm{mt}(g)+1,$$
and hence $\gamma(f)=\kappa(f)-\mathrm{mt}(f)+1,$ due to Remark \ref{rm21}(e). This completes the proposition.
\end{proof}
%Notice that the inverse assertions for equalities are not true in general as the following example shows. Take$f=x^6+x^7-y^9\in K[[x,y]]$ with $\mathrm{char}(K)=3$. Then $f$ is irreducible and has a parametrization  $\left(x(t),y(t)\right)=(t^9,t^6+t^7)$.
The following proposition says that the number $\tilde{\gamma}$ is invariant under contact equivalence in the class of singularities for which $p$ is m-good. It will be shown in Corollary \ref{coro21} that $\tilde{\gamma}$ is invariant under contact equivalence in the class of singularities for which $p$ is im-good.
\begin{proposition}\label{pro22}
Let $f\in K[[x,y]]$ be reduced such that $p$ is m-good for $f$ and let $g\sim_c f$. Then $\tilde{\gamma}(g)=\tilde{\gamma}(f)$. In particular, $\gamma(f)=\tilde{\gamma}(f)$.
\end{proposition}
\begin{proof}
This follows from Proposition \ref{pro21} and Remark \ref{rm21}(e). See \cite[Lemma 2.3.4]{Ng13} for a direct proof.
\end{proof}
For the proof of the main results of this section (Theorem \ref{thm21} and \ref{thm22}) we need the following two technical lemmas, which compare the gamma invariants after blowing ups. %Let us recall that the blowing up of the affine plane with center $O = (0, 0)$ is the surface $\tilde{X}$ in $\mathbb A^2 \times \mathbb P^1$ with defining equation $xv = yu$, where $(x, y )$ are the coordinates of $\mathbb A^2$ and $(u : v )$ are the homogeneous coordinates of $\mathbb P^1$. Consider the two charts of $\mathbb A^2 \times \mathbb P^1$ defined by $u\neq 0$, $v\neq 0$ respectively. 
Let $f\in K[[x,y]]$ be an irreducible plane curve singularity and $(\beta:\alpha)\in \mathbb P^1$ be its {\em tangent direction} (cf. \cite{Cam80}, Lemma 3.4.5 and  \cite{GLS06}, Chap. 1, Lemma 3.19), i.e. $i(f, \alpha x-\beta y)>\mathrm{mt}(f)$. Assume that $(\beta:\alpha)$ is a point in the first chart of $\mathbb P^1$. Then the strict transform of $f$ is a formal series $\tilde{f}$ defined by $f(u,u(v+\frac{\alpha}{\beta}))=u^m\tilde{f}(u,v)$ (One can define $\tilde{f}$ similarly if $(\beta:\alpha)$ belongs to the second chart). 
\begin{lemma}\label{lm24}
Let $f\in K[[x,y]]$ be irreducible such that $m:=i(f,x)=i(f,y)$. Let $g\in K[[x,y]]$ be such that $f(x,y)=g(x,\alpha x-\beta y)$, where $(\beta:\alpha)\in \mathbb P^1$ is the unique tangent direction of $f$. Then 
\begin{itemize}
\item[(i)] $m=i(g,x)<i(g,y)$.
\item[(ii)] $\tilde{\gamma}(f)\geq \tilde{\gamma}(g)$.
\item[(iii)] If the characteristic $p$ is im-good for $g$ but not for $f$, then $$\tilde{\gamma}(f)>\tilde{\gamma}(g).$$ 
\end{itemize}
\end{lemma}
\begin{proof}
(i) It follows from Remark \ref{rm21} that
$$i(g,x)=i\big (g(x,\alpha x-\beta y),x\big )=i(f,x)=m$$
and
$$i(g,y)=i\big (g(x,\alpha x-\beta y),\alpha x-\beta y\big )=i(f,\alpha x-\beta y)>m,$$
which proves (i).

(ii) Let $(x(t),y(t))$ be a parametrization of $f$. Then
\begin{eqnarray*}
X(t)&=& x(t)\\
Y(t)&=& \alpha x(t)-\beta y(t)
\end{eqnarray*}
is a parametrization of $g$. Since $f(x,y)=g(x,\alpha x-\beta y)$,
\begin{eqnarray*}
f_x(x,y)&=& g_x(x,\alpha x-\beta y)+\alpha g_y(x,\alpha x-\beta y)\\
f_y(x,y)&=& -\beta g_y(x,\alpha x-\beta y)
\end{eqnarray*}
and therefore
\begin{eqnarray*}
f_x(x(t),y(t))&=& g_x(X(t),Y(t))+\alpha g_y(X(t),Y(t))\\
f_y(x(t),y(t))&=& -\beta g_y(X(t),Y(t)).
\end{eqnarray*}
We consider the two following cases:

$\bullet$ If $i(f,f_x)\geq i(f,f_y)$. Then 
\begin{eqnarray*}
\tilde{\gamma}(f) &=&\min\{i(f,f_x)-i(f,y)+1, i(f,f_y)-i(f,x)+1\}\\
&=& i(f,f_y)-m+1\\
&=& i(g,g_y)-i(g,x)+1\\
&\geq & \tilde{\gamma}(g).
\end{eqnarray*}

$\bullet$ If $i(f,f_x)< i(f,f_y)$, then $\mathrm{ord} f_x(x(t),y(t))<\mathrm{ord} f_y(x(t),y(t))=g_y(X(t),Y(t))$. This, together with the equality $f_x(x(t),y(t))=g_x(X(t),Y(t))+\alpha g_y(X(t),Y(t))$ implies that
$$\mathrm{ord} f_x(x(t),y(t))=\mathrm{ord} g_x(X(t),Y(t))<\mathrm{ord} g_y(X(t),Y(t)),$$
or equivalently $i(f,f_x)=i(g,g_x)< i(g,g_y)$. Hence
\begin{eqnarray*}
\tilde{\gamma}(g) &=&\min\{i(g,g_x)-i(g,y)+1, i(g,g_y)-i(g,x)+1\}\\
&=& i(g,g_x)-i(g,y)+1\\
&<& i(f,f_x)-i(g,x)+1\\
&=& i(f,f_x)-m+1\\
&=& \tilde{\gamma}(f).
\end{eqnarray*}

(iii) As in the proof of part (ii), if $i(f,f_x)< i(f,f_y)$ then $\tilde{\gamma}(f)>\tilde{\gamma}(g)$. Assume now that $i(f,f_x)\geq i(f,f_y)$. Then as above, we have
$$\tilde{\gamma}(f)= i(g,g_y)-i(g,x)+1.$$
Since $p$ is not im-good for $f$, $m= 0$ (mod $p$) and therefore $i(g,y) \neq 0$ (mod $p$) since $p$ is im-good for $g$. This, together with the equalities $\text{ord } Y(t)=i(g,y)$ and $\text{ord } X(t)=i(g,x)=m$ implies that 
$$\text{ord }\dot Y(t)=i(g,y)-1 \text{ and }\text{ord }\dot X(t)>m-1=i(g,x)-1.$$
On the other hand, since $g(X(t),Y(t))=0$, we have
$$\dot X(t)\cdot g_x(X(t),Y(t))+\dot Y(t)\cdot g_y(X(t),Y(t))=0.$$
It yields
$$\text{ord }\dot X(t)+\text{ord }g_x(X(t),Y(t))=\text{ord }\dot Y(t)+\text{ord }g_y(X(t),Y(t)),$$
or, equivalently
$$i(g,g_x)-\text{ord }\dot Y(t)=i(g,g_y)-\text{ord }\dot X(t).$$
This implies that
$$i(g,g_x)-i(g,y)<i(g,g_y)-i(g,x).$$
Hence 
\begin{eqnarray*}
\tilde{\gamma}(g) &=&\min\{i(g,g_x)-i(g,y)+1, i(g,g_y)-i(g,x)+1\}\\
&=& i(g,g_x)-i(g,y)+1\\
&<& i(g,g_y)-i(g,x)+1\\
&=& \tilde{\gamma}(f).
\end{eqnarray*}
\end{proof} 
\begin{lemma}\label{lm25}
Let $f\in K[[x,y]]$ be irreducible and $\tilde{f}$ its strict transform, then 
$$\tilde{\gamma}(f)\geq m^2-m+\tilde{\gamma}(\tilde{f}).$$
Assume moreover that $i(f,x)\neq i(f,y)$. Then 
\begin{itemize}
\item[(i)]  $\tilde{\gamma}(f)= m^2-m+\tilde{\gamma}(\tilde{f})$, with $m:=\mathrm{mt}(f)$ the multiplicity of $f$.
\item[(ii)] $p$ is im-good for $f$, if and only if it is so for $\tilde{f}$.
\end{itemize}
\end{lemma}
\begin{proof}
(i) If $f$ is not convenient then either $f=x\cdot u$ or $f=y\cdot u$ for some unit $u$ since $f$ is irreducible and hence the lemma is evident.

Assume now that $f$ is convenient and that $i(f,x)< i(f,y)$. Then the (local equation of) $\tilde{f}$ at the point $(1:0)$ in the first chart is: 
$$f(u,uv)=u^m \tilde{f}(u,v)$$
and therefore 
\begin{eqnarray*}
f_x(u,uv)+vf_y(u,uv)&=&m u^{m-1} \tilde{f}(u,v)+u^{m} \tilde{f_u}(u,v)\\
uf_y(u,uv)&=&u^{m}\tilde{f_v}(u,v).
\end{eqnarray*}
It yields 
\begin{eqnarray*}
xf_x(x,y)+yf_y(x,y)&=&m u^{m} \tilde{f}(u,v)+u^{m} \big (u\tilde{f_u}(u,v)\big )\\
yf_y(x,y)&=&u^{m} \big (v\tilde{f_v}(u,v)\big ),
\end{eqnarray*}
where $x=u, y=uv$.

Take a parametrization $(u(t),v(t))$ of $\tilde{f}$. Then 
\begin{eqnarray*}
x(t)&=&u(t)\\
y(t)&=&u(t)v(t)
\end{eqnarray*}
will be a parametrization of $f$ and 
\begin{eqnarray*}
x(t)f_x(x(t),y(t))+y(t)f_y(x(t),y(t))&=&u(t)^{m} \big (u(t)\tilde{f_u}(u(t),v(t))\big )\\
y(t)f_y(x(t),y(t))&=&u(t)^{m} \big (v(t)\tilde{f_v}(u(t),v(t))\big ).
\end{eqnarray*}
Thus
$$\alpha x(t)f_x(x(t),y(t))+(\alpha+\beta)y(t)f_y(x(t),y(t))=\alpha u(t)^{m}\tilde{f_u}(u(t),v(t))+\beta v(t)^{m}\tilde{f_v}(u(t),v(t)),$$
for $(\alpha:\beta)\in \mathbb P^1$ generic. It follows that
$$i\big(f,\alpha xf_x+(\alpha+\beta)yf_y\big)=m^2+i\big(\tilde f,\alpha u\tilde{f_u}+\beta v\tilde{f_v}\big)$$
since $\text{ord } u(t)=m$. Besides, 
\begin{eqnarray*}
i(f,x)+i(f,y)&=&\mathrm{ord }x(t)+\mathrm{ord }y(t)\\
&=&\mathrm{ord } u(t)+\mathrm{ord }u(t)+\mathrm{ord }v(t)\\
&=& m+\mathrm{ord }u(t)+\mathrm{ord }v(t)\\
&=& m+i(\tilde{f},u)+i(\tilde{f},v).
\end{eqnarray*}
Hence by Remark \ref{rm22}(c) we have
\begin{eqnarray*}
\tilde{\gamma}(f)&=&i(f,\alpha xf_x+(\alpha+\beta) yf_y)-i(f,x)-i(f,y)+1\\
&=& m^2-m+i(\tilde f,\alpha u\tilde{f_u}+\beta v\tilde{f_v})-i(\tilde{f},u)-i(\tilde{f},v)+1\\
&=& m^2-m+\tilde{\gamma} (\tilde{f}).
\end{eqnarray*}

(ii) follows from the equalities 
$$i(f,x)=\text{ord }x(t)=\text{ord } u(t)=i(\tilde{f},u)$$
and
$$i(f,y)=\text{ord }y(t)=\text{ord } u(t)+\text{ord } v(t)=i(\tilde{f},u)+i(\tilde{f},v).$$

In general, it is sufficient to prove $\tilde{\gamma}(f)\geq m^2-m+\tilde{\gamma} (\tilde{f})$ for provided $i(f,x)=i(f,y)$. Let $(\beta:\alpha)$ be the unique tangent direction of $f$ and $g\in K[[x,y]]$ such that $f(x,y)=g(x,\alpha x-\beta y)$. Then by Lemma \ref{lm24}, $i(g,x)<i(g,y)$ and $\tilde{\gamma}(f)\geq \tilde{\gamma}(g)$. It follows from (i) that 
$$\tilde{\gamma}(g)= m^2-m+\tilde{\gamma}(\tilde{g}),$$ where $\tilde{g}$ is the strict transform of $g$.

Besides, it is easy to see that the local equation of $\tilde{f}$ at the point $(\beta:\alpha)$ coincides with that of $\tilde{g}$ at the point $(1:0)$. This means that $\tilde{\gamma}(\tilde{f})=\tilde{\gamma}(\tilde{g})$. Hence
$$\tilde{\gamma}(f)\geq \tilde{\gamma}(g)\geq m^2-m+\tilde{\gamma}(\tilde{g})=m^2-m+\tilde{\gamma}(\tilde{f}).$$ 
\end{proof}
Note that the delta invariants admit an analogous property. More precisely,
\begin{remark}\label{rm23}
{\rm For each reduced (not necessary irreducible) plane curve singularity $f\in K[[x,y]]$ one may define the notion of strict transform $\tilde{f}$ of $f$ and get the following formula (cf. \cite{GLS06}, Chap. I, Prop. 3.34)
$$2\delta(f)=m^2-m+2\delta(\tilde{f}).$$
However a similar relation for kappa invariants is unknown.
}\end{remark}
\begin{theorem}\label{thm21}
Let $f\in K[[x,y]]$ be reduced. Then
$$\tilde{\gamma}(f)\geq 2\delta(f)-r(f)+1.$$
Equality holds if and only if the characteristic $p$ is im-good for $f$.
\end{theorem}
\begin{proof}
The proof will be divided into two steps\\
{\bf Step 1:} $f$ is irreducible. We argue by induction on the delta invariant of $f$. \\
$\bullet$ {\it Inequality:} If $\delta(f)=0$, i.e. $f$ is non-singular and then $\tilde{\gamma}(f)=0$. Suppose that $\delta(f)>0$ and the theorem is true for any $g$ satisfying $\delta(g)<\delta(f)$.  It follows from Remark \ref{rm23} that
$$\delta(f)=\frac{m(m-1)}{2}+\delta(\tilde{f})>\delta(\tilde{f}).$$
Applying the induction hypothesis to $\tilde{f}$ we obtain 
\begin{eqnarray*}
\tilde{\gamma} (f)&\geq & m^2-m+\tilde{\gamma} (\tilde{f})\\
&\geq& m^2-m+2 \delta(\tilde{f})\\
&=&2 \delta(f)
\end{eqnarray*}
due to Lemma \ref{lm25} and Remark \ref{rm23}. This proves the inequality of the theorem.\\
$\bullet$ {\it ``if'' statement:} Assume now that $p$ is im-good for $f$. We need to show that $\tilde{\gamma} (f)=2 \delta(f)$.

- If $i(f,x)\neq i(f,y)$ then $\tilde{\gamma}(f)=\tilde{\gamma}(\tilde{f})$ and $p$ is also im-good for $\tilde{f}$ by Lemma \ref{lm25}. By induction hypothesis, $\tilde{\gamma} (\tilde{f})= 2 \delta(\tilde{f})$. It hence follows from Lemma \ref{lm25} and Remark \ref{rm23} that
\begin{eqnarray*}
\tilde{\gamma} (f)& = & m^2-m+\tilde{\gamma} (\tilde{f})\\
&=& m^2-m+2 \delta(\tilde{f})\\
&=&2 \delta(f).
\end{eqnarray*}

- If $i(f,x)=i(f,y)$, then $i(f,x)=i(f,y)=m$ and therefore $m\neq 0$ (mod $p$) by assumption that $p$ is im-good for $f$. Take $g\in K[[x,y]]$ as in Lemma \ref{lm24} then $\tilde{\gamma}(f)=\tilde{\gamma}(g)$ by Proposition \ref{pro22} and $\delta(f)=\delta(g)$ by Remark \ref{rm21}. Applying induction hypothesis to the strict transform $\tilde{g}$ of $g$ gives $\tilde{\gamma}(\tilde{g})=2 \delta(\tilde{g})$. Combining Lemma \ref{lm25} and Remark \ref{rm23} we get
\begin{eqnarray*}
\tilde{\gamma} (f)=\tilde{\gamma}(g)& = & m^2-m+\tilde{\gamma} (\tilde{g})\\
&=& m^2-m+2 \delta(\tilde{g})\\
&=&2 \delta(g) =2\delta(f).
\end{eqnarray*}
This proves the sufficiency of the equality.\\
$\bullet$ {\it ``only if'' statement:} Finally, we will prove that $\tilde{\gamma}(f)>2\delta(f)$ if $p$ is not im-good for $f$ by induction on the delta invariant of $f$. Since $p$ is not im-good for $f$, $m\geq p$ and hence $\delta(f)\geq p(p-1)/2$. 

- If $\delta(f)= p(p-1)/2$, then $m=p$ and $\mathrm{mt}(\tilde f)=1$. We may write
$$f=f_p+f_{p+1}+\ldots,$$
where $f_p=(\alpha x-\beta y)^p$ with $\beta\neq 0$. We shall show that $\alpha\neq 0$. By contradiction, suppose that $\alpha=0$. Then $i(f,y)>p=i(f,x)$. Besides $i(f,y)=0$ (mod $p$) since $p$ is not im-good for $f$ and therefore $i(f,y)\geq 2p$. Thus it is easy to see that $\mathrm{mt}(\tilde f)\geq p>1$, which is a contradiction and hence $\alpha\neq 0$. Take $g\in K[[x,y]]$ as in Lemma \ref{lm24}, then $p=i(g,x)<i(g,y)$.

On the other hand, $\tilde{g}$ must be non-singular (i.e. $\mathrm{mt}(\tilde g)=1$) since 
$$p(p-1)/2=\delta(f)=\delta(g)=p(p-1)/2+\delta(\tilde{g}).$$
This implies that $i(\tilde{g},v)=1$. Hence
$$i(g,y)=i(\tilde{g},v)+i(g,x)=p+1.$$
Consequently, $p$ is im-good for $g$ and therefore $\tilde{\gamma}(f)>\tilde{\gamma}(g)$ by Lemma \ref{lm24}. Applying the first part to $g$ we have $\tilde{\gamma}(g)\geq 2 \delta(g)$ and hence
$$\tilde{\gamma}(f)>\tilde{\gamma}(g)\geq 2 \delta(g)=2\delta(f).$$

- Now we prove the induction step. Assume that $\delta(f)>p(p-1)/2$. If $i(f,x)\neq i(f,y)$ then $p$ is not im-good for $\tilde{f}$ by Lemma \ref{lm25} since it is not im-good for $f$. We can apply the induction hypothesis to $\tilde{f}$ and obtain
\begin{eqnarray*}
\tilde{\gamma}(f)&=&m(m-1)+\tilde{\gamma}(\tilde{f})\\
&>& m(m-1)+2\delta(\tilde{f})\\
&=& 2\delta(f).
\end{eqnarray*}

Assume that $i(f,x)= i(f,y)$. Take $g\in K[[x,y]]$ as in Lemma \ref{lm24}. If $p$ is not im-good for $g$, since $i(g,x)\neq i(g,y)$, we may apply the above argument, with $f$ replaced by $g$, to obtain $\tilde{\gamma}(g)>2\delta(g)$ and hence
$$\tilde{\gamma}(f)\geq\tilde{\gamma}(g)>2\delta(g)=2 \delta(f),$$
where equalities follow from Lemma \ref{lm24} and Remark \ref{rm21}. If $p$ is im-good for $g$, then $\tilde{\gamma}(f)>\tilde{\gamma}(g)$ by Lemma \ref{lm24} and therefore
$$\tilde{\gamma}(f)>\tilde{\gamma}(g)\geq 2\delta(g)=2 \delta(f).$$
This proves the first step.\\
{\bf Step 2:} Assume that $f$ decomposes into its branches $f=f_1\cdot\ldots\cdot f_r$. Then 
$$\tilde{\gamma} (f)=\sum_{i=1}^r\big ( \tilde{\gamma}(f_i)+\sum_{j\neq i} i(f_i,f_j)\big )-r+1$$
and
$$2\delta (f)=\sum_{i=1}^r\big ( 2\delta(f_i)+\sum_{j\neq i} i(f_i,f_j)\big ).$$
The proposition follows from the above equalities and Step 1.
\end{proof}
\begin{corollary}\label{coro21}
Assume that $p$ is im-good for $f$. Then
$$\gamma(f)=\tilde{\gamma}(f).$$
\end{corollary}
\begin{proof}
Let $g$ be right equivalent to $f$ such that $\gamma(f)=\tilde{\gamma}(g)$. It then follows from Theorem \ref{thm21} and Remark \ref{rm21} that
\begin{eqnarray*}
\tilde{\gamma}(f)\geq \gamma(f) =\tilde{\gamma}(g)&\geq & 2\delta(g)-r(g)+1 = 2\delta(f)-r(f)+1 =\tilde{\gamma}(f),
\end{eqnarray*}
and hence $\gamma(f)=\tilde{\gamma}(f)$.
\end{proof}
The following simple corollary should be useful in computation, since the number in the left side is easily computed. 
\begin{corollary}\label{coro22}
Assume that $p>\mathrm{mt}(f)$. Then
$$\mu(f)-\tilde{\gamma}(f)=\mathrm{Sw}(f).$$
\end{corollary}
\begin{theorem}\label{thm22}
Let $f\in K[[x,y]]$ be reduced. Then
$$\gamma(f)\geq 2\delta(f)-r(f)+1.$$
Equality holds if and only if the characteristic $p$ is right im-good for $f$.
\end{theorem}
\begin{proof}
Taking $g$ right equivalent to $f$ such that $\gamma(f)=\tilde{\gamma}(g)$ and combining Theorem \ref{thm21} and Remark \ref{rm21} we get
$$\gamma(f)=\tilde{\gamma}(g)\geq 2\delta(g)-r(g)+1=2\delta(f)-r(f)+1$$
with equality if and only if $p$ is im-good for $g$. It remains to show that if $p$ is right im-good for $f$, then 
$$\gamma(f)=2\delta(f)-r(f)+1.$$
Indeed, by definition, $p$ is im-good for some $h$ right equivalent to $f$. Again combining Theorem \ref{thm21} and Remark \ref{rm21} we get
$$\gamma(f)=\gamma(h)\leq \tilde{\gamma}(h)= 2\delta(h)-r(h)+1=2\delta(f)-r(f)+1\leq \gamma(f).$$
This implies that
$$\gamma(f)=2\delta(f)-r(f)+1,$$
which completes the theorem.
\end{proof}
\section{\bf Kappa invariants and Pl\"ucker formulas}\label{sec3}
We prove in this section the main result (Theorem \ref{thm31}) and apply it to Pl\"ucker formulas (Corollaries \ref{coro34}, \ref{coro35}). Furthermore we show, in Corollary \ref{coro33} (resp. Corollary \ref{coro35}), that if $p$ is ``big'' for $f$ (resp. for a plane curve $C$), then $f$ (resp. $C$) has no wild vanishing cycle.
\begin{theorem}\label{thm31}
Let $f\in K[[x,y]]$ be reduced. One has  
$$\kappa(f)\geq 2\delta(f)+\mathrm{mt}(f)-r(f)$$
with equality if and only if $p$ is m-good for $f$.
\end{theorem}
\begin{proof}
Combining Proposition \ref{pro21} and Theorem \ref{thm21} we get
$$\kappa(f)\geq \tilde{\gamma}(f)+\mathrm{mt}(f)-1\geq 2\delta(f)+\mathrm{mt}(f)-r(f),$$
with equalities if $p$ is m-good for $f$. It then remains to prove that if $p$ is not m-good for $f$ then  
$$\kappa(f)> 2\delta(f)+\mathrm{mt}(f)-r(f).$$
It suffices to prove the inequality for $p$ which is im-good for $f$, since otherwise we have 
$$\tilde{\gamma}(f)> 2\delta(f)-r(f)+1$$
due to Theorem \ref{thm21}, and hence
$$\kappa(f)> 2\delta(f)+\mathrm{mt}(f)-r(f).$$
Suppose that $p$ is im-good for $f$. By Remark \ref{rm22} we may assume that $f$ is irreducible. Without loss of generality we may assume further that $i(f,x)\leq i(f,y)$. Then $m:=\mathrm{mt}(f)=i(f,x)$ and therefore $m=0$, $i(f,y)\neq 0$ (mod $p$) since $p$ is not m-good but im-good for $f$. It yields that $i(f,x)< i(f,y)$. Putting $g(x,y)=f(x,y-x)$ and applying Lemma \ref{lm24} with replacing the role of $f$ and $g$ we obtain that $\tilde{\gamma}(g)>\tilde{\gamma}(f)$. Hence
\begin{eqnarray*}
\kappa(f)=\kappa(g)&\geq& \tilde{\gamma}(g)+\mathrm{mt}(g)-1\\
&>& \tilde{\gamma}(f)+\mathrm{mt}(f)-1\\
&\geq & 2\delta(f)+\mathrm{mt}(f)-r(f)
\end{eqnarray*}
by combining Remark \ref{rm21}, Proposition \ref{pro21} and Theorem \ref{thm21}.
\end{proof}
The following interesting corollary says that if the characteristic $p$ is ``big'' for $f$, then $f$ has no wild vanishing cycle. 
\begin{corollary}\label{coro33}
Assume that $p>\kappa(f)$. Then $f$ has no wild vanishing cycle, i.e. $\mathrm{Sw}(f)=0$. Moreover one has 
\begin{eqnarray*}
\kappa(f)&=& 2\delta(f)+\mathrm{mt}(f)-r(f)\\
&=&\mu(f)+\mathrm{mt}(f)-1.
\end{eqnarray*} 
\end{corollary}
\begin{proof}
Clearly
$$\kappa(f)\geq\mathrm{mt}(f)$$
and then $p>\mathrm{mt}(f)$. Therefore $p$ is m-good for $f$ and then 
$$\kappa(f)= 2\delta(f)+\mathrm{mt}(f)-r(f),$$
due to Theorem \ref{thm31}. It thus suffices to show that 
$$\kappa(f)=\mu(f)+\mathrm{mt}(f)-1.$$
Indeed, take $(\alpha:\beta), (a:b)\in \mathbb P^1$ such that $\alpha\cdot b-\beta\cdot a\neq 0$ and that
$$\kappa(f)=i(f,\alpha f_x+\beta f_y),\text{ and } i(g,x)=i(g,y)=\mathrm{mt}(g)=\mathrm{mt}(f)$$
with $g(x,y):=f(\alpha x+ a y, \beta x+b y)$. Let $g_x=g_1\cdot\ldots\cdot g_s$ with $g_i$ irreducible and let $(x_i(t),y_i(t))$ be a parametrization of $g_i$. Since 
$$p>\kappa(f)\geq\mathrm{ord}\left(g\left(x_i(t),y_i(t)\right)\right) \text{ and } p>\mathrm{mt}(g)>\mathrm{ord}\left( y_i(t)\right),$$
it yields that
\begin{eqnarray*}
\mathrm{ord}\left(g\left(x_i(t),y_i(t)\right)\right)&=&\mathrm{ord}\left(\frac{d}{dt} g\left(x_i(t),y_i(t)\right)\right)+1\\
&=&\mathrm{ord}\left(g_y\left(x_i(t),y_i(t)\right)\right)+\mathrm{ord}\left(\frac{d}{dt} y_i(t)\right) +1\\
&=& \mathrm{ord}\left(g_y\left(x_i(t),y_i(t)\right)\right)+\mathrm{ord}\left( y_i(t)\right).
\end{eqnarray*}
This implies, by the additivity of intersection multiplicities, that
\begin{eqnarray*}
i(g,g_x)&=&i(g_x,g_y)+i(g_x,y)\\
&=& \mu(g)+\mathrm{mt}(g)-1.
\end{eqnarray*}
Hence, by Remark \ref{rm21},
\begin{eqnarray*}
\kappa(f)=i(f,\alpha f_x+\beta f_y)&=&i(g,g_x)\\
&=& \mu(g)+\mathrm{mt}(g)-1\\
&=& \mu(f)+\mathrm{mt}(f)-1,
\end{eqnarray*}
which finishes the corollary.
\end{proof}
%Moreover, we have the following interesting result regarding the Pl\"ucker formulas (see \cite{BK86}, \cite{Pie78}). 
%Applying the main result we obtain the following Pl\"ucker formulas. 
Let $C$ be a irreducible curve of degree $d$ in $\mathbb P^2$ defined by a homogeneous polynomial $F\in K[x,y,z]$. Let $\mathrm{Sing}(C)$ resp. $C^{*}:=C\setminus \mathrm{Sing}(C)$ the singular resp. smooth locus of $C$, and let $s(C):=\sharp \mathrm{Sing}(C)$ the number of singular points. Let $\rho\colon C^{*}\to \check{\mathbb P}^2, P=(x\colon y\colon z)\mapsto \left(F_{x}(P)\colon F_{y}(P)\colon F_{z}(P)\right)$ the dual (Gauss) map and $\deg(\rho)$ its degree. We call the closure of the image of $\rho$ in $\check{\mathbb P}^2$ the {\em dual curve} of $C$ denoted by $\check{C}$. We denote by $\check{d}$ the degree of $\check{C}$. For each singular point $P\in \mathrm{Sing}(C)$ take a local equation $f_P=0$ of $C$ at $P$, and define
\begin{displaymath}
\begin{array}{lclcll}
\delta(C)&:=& \sum\delta(f_P), & \mathrm{mt}(C)&:=& \sum\mathrm{mt}(f_P),\\
\mu(C)&:=& \sum\mu(f_P), & r(C)&:=& \sum r(f_P),\\
\mathrm{Sw}(C)&:=& \sum \mathrm{Sw}(f_P). &
\end{array}
\end{displaymath}
where all the sums are taken over $P\in \mathrm{Sing}(C)$.
\begin{proof}
Clearly
$$\kappa(f)\geq\mathrm{mt}(f)$$
and then $p>\mathrm{mt}(f)$. Therefore $p$ is m-good for $f$ and then 
$$\kappa(f)= 2\delta(f)+\mathrm{mt}(f)-r(f),$$
due to Theorem \ref{thm31}. It thus suffices to show that 
$$\kappa(f)=\mu(f)+\mathrm{mt}(f)-1.$$
Indeed, take $(\alpha:\beta), (a:b)\in \mathbb P^1$ such that $\alpha\cdot b-\beta\cdot a\neq 0$ and that
$$\kappa(f)=i(f,\alpha f_x+\beta f_y),\text{ and } i(g,x)=i(g,y)=\mathrm{mt}(g)=\mathrm{mt}(f)$$
with $g(x,y):=f(\alpha x+ a y, \beta x+b y)$. Let $g_x=g_1\cdot\ldots\cdot g_s$ with $g_i$ irreducible and let $(x_i(t),y_i(t))$ be a parametrization of $g_i$. Since 
$$p>\kappa(f)\geq\mathrm{ord}\left(g\left(x_i(t),y_i(t)\right)\right) \text{ and } p>\mathrm{mt}(g)>\mathrm{ord}\left( y_i(t)\right),$$
it yields that
\begin{eqnarray*}
\mathrm{ord}\left(g\left(x_i(t),y_i(t)\right)\right)&=&\mathrm{ord}\left(\frac{d}{dt} g\left(x_i(t),y_i(t)\right)\right)+1\\
&=&\mathrm{ord}\left(g_y\left(x_i(t),y_i(t)\right)\right)+\mathrm{ord}\left(\frac{d}{dt} y_i(t)\right) +1\\
&=& \mathrm{ord}\left(g_y\left(x_i(t),y_i(t)\right)\right)+\mathrm{ord}\left( y_i(t)\right).
\end{eqnarray*}
This implies, by the additivity of intersection multiplicities, that
\begin{eqnarray*}
i(g,g_x)&=&i(g_x,g_y)+i(g_x,y)\\
&=& \mu(g)+\mathrm{mt}(g)-1.
\end{eqnarray*}
Hence, by Remark \ref{rm21},
\begin{eqnarray*}
\kappa(f)=i(f,\alpha f_x+\beta f_y)&=&i(g,g_x)\\
&=& \mu(g)+\mathrm{mt}(g)-1\\
&=& \mu(f)+\mathrm{mt}(f)-1,
\end{eqnarray*}
which finishes the corollary.
\end{proof}
\begin{corollary}\label{coro34}
Using the above notions, we have
\begin{eqnarray*}
\deg(\rho)\cdot\check{d}&\leq& d(d-1)-2\delta(C)+r(C)-\mathrm{mt}(C)\\
&=& d(d-1)-\mu(C)-\mathrm{mt}(C)+s(C)+\mathrm{Sw}(C),
\end{eqnarray*}
with equality if and only if $p$ is {\em multiplicity good (m-good)} for $C$, i.e. $p$ is m-good for all the $f_P$. 
\end{corollary}
\begin{proof}
The following formula is known as the first Pl\"ucker formula in positive characteristic,
\begin{eqnarray}\label{eq31}
\deg(\rho)\cdot\check{d}= d(d-1)-\sum_{P\in \mathrm{Sing}(C)}\kappa(f_P).
\end{eqnarray}
However we can not find an exact reference, so for the convenience of the reader, we give a short proof. We denote by $\mathrm{sdeg}(\rho)$ (resp. $\mathrm{ideg}(\rho)$) the separable (resp. inseparable) degree of $\rho$. Then there exists an open subset $V\subset \rho(C)$ such that 
$$\sharp\ \rho^{-1}(R)=\mathrm{sdeg}(\rho)\ \text{ for all } R\in V.$$ 
It is easy to see that there exists an open subset $U\subset C$ such that 
$$H_Q\cap \check{C}\subset V\ \text{ for all } Q\in U,$$
where for each point $Q=\left(\alpha\colon\beta\colon\gamma\right)\in \mathbb P^2$, $H_Q$ denotes the line in $\check{\mathbb P}^2$ defined by $\alpha X+\beta Y+\gamma Z$. Moreover it follows from the ramification theory that
$$i_P\left(C,P_Q\right)=\mathrm{ideg}(\rho)\cdot i_{\rho(P)}\left(\check{C},H_Q\right) \ \text{ for all } P\in C^{*},$$
where $P_Q$ denotes the polar curve of $C$ w.r.t. $Q$ defined by $\alpha F_x+\beta F_y+\gamma F_z$, and $i_P\left(C,P_Q\right)$ the intersection multiplicity of $C$ and $P_Q$ at $P$. Hence from B\'ezout theorem we have, with $Q$ generic,
\begin{eqnarray*}
d(d-1)&=&\sum_{P\in C} i_P\left(C,P_Q\right)\\
&=&\sum_{P\in \mathrm{Sing}(C)} i_P\left(C,P_Q\right)+\sum_{P\in C^{*}} i_P\left(C,P_Q\right)\\
&=&\sum_{P\in \mathrm{Sing}(C)}\kappa(f_P)+\mathrm{ideg}(\rho)\sum_{P\in C^{*}} i_{\rho(P)}\left(\check{C},H_Q\right)\\
&=&\sum_{P\in \mathrm{Sing}(C)}\kappa(f_P)+\mathrm{ideg}(\rho)\cdot \mathrm{sdeg}(\rho)\sum_{R\in \check{C}} i_{R}\left(\check{C},H_Q\right)\\
&=&\sum_{P\in \mathrm{Sing}(C)}\kappa(f_P)+ \deg(\rho)\cdot\check{d}.
\end{eqnarray*}
This completes the first Pl\"ucker formula. The corollary hence follows from Theorem \ref{thm31}.
\end{proof}
Combining Corollary \ref{coro33} and (\ref{eq31}) we obtain 
\begin{corollary}\label{coro35}
With the above notions, assume that 
$$\max_{P\in \mathrm{Sing}(C)}\{\kappa(f_P)\}<p,$$
(for example, $d(d-1)<p$). Then $C$ has no wild vanishing cycle, i.e. $\mathrm{Sw}(C)=0$. Moreover one has
\begin{eqnarray*}
\deg(\rho)\cdot\check{d}&=& d(d-1)-2\delta(C)+r(C)-\mathrm{mt}(C)\\
&=& d(d-1)-\mu(C)-\mathrm{mt}(C)+s(C).
\end{eqnarray*}
\end{corollary}

%\begin{corollary}\label{cor2.3}
%Let $f$ be weakly quasihomogeneous. Then 
%$$\mu(f)\geq 2\delta(f)-r(f)+1$$
%with equality iff $p$ is im-good for $f$.
%\end{corollary}

%So far we relate gamma invariant to delta invariant, kappa invariant and the number of branches. But we can not say any thing about the relationship between gamma invariant and Milnor number. 
%\begin{conjecture}\label{conj2.4.1.1}
%The following are equivalent
%\begin{itemize}
%\item[(i)] $\mu(f)=2\delta(f)-r(f)+1$,
%\item[(ii)] $p$ is good for $f$, i.e. $p$ is does not divide all $i(f_i,x),i(f_i,y)$ for any branch $f_i$ of $f$.  
%\end{itemize}
%\end{conjecture}

\end{document}